\RequirePackage{fix-cm}
\documentclass[smallextended]{svjour3}       
\smartqed  
\usepackage{graphicx}

\usepackage{algorithm}
\usepackage{algorithmic}

\usepackage{subfigure}
\usepackage{float}

\usepackage{amsfonts}

\usepackage{verbatim}

\usepackage{amssymb}
\begin{document}

\title{Modified $\l_{p}$-norm regularization minimization for sparse signal recovery}


\author{Angang Cui \and
        Jigen Peng\and
        Haiyang Li
}

\institute{Angang Cui  \at School of Mathematics and Statistics, Xi'an Jiaotong University, Xi'an, 710049, China. \email{cuiangang@163.com}\and Jigen Peng \at School of Mathematics and Statistics, Xi'an Jiaotong University, Xi'an, 710049, China. Corresponding author. \email{jgpengxjtu@126.com}\and Haiyang Li\at School of Science, Xi'an Polytechnic University, Xi'an, 710048, China. \email{fplihaiyang@126.com}}

\date{Received: date / Accepted: date}

\maketitle

\begin{abstract}
In numerous substitution models for the $\l_{0}$-norm minimization problem $(P_{0})$, the $\l_{p}$-norm minimization $(P_{p})$ with $0<p<1$ have been considered as the most natural choice.
However, the non-convex optimization problem $(P_{p})$ are much more computational challenges, and are also NP-hard. Meanwhile, the algorithms corresponding to the proximal mapping of the
regularization $\l_{p}$-norm minimization $(P_{p}^{\lambda})$ are limited to few specific values of parameter $p$. In this paper, we replace the $\ell_{p}$-norm $\|x\|_{p}^{p}$ with a
modified function $\sum_{i=1}^{n}\frac{|x_{i}|}{(|x_{i}|+\epsilon_{i})^{1-p}}$. With change the parameter $\epsilon>0$, this modified function would like to interpolate the
$\l_{p}$-norm $\|x\|_{p}^{p}$. By this transformation, we translated the $\l_{p}$-norm regularization minimization $(P_{p}^{\lambda})$ into a modified $\l_{p}$-norm regularization minimization $(P_{p}^{\lambda,\epsilon})$. Then, we develop the thresholding representation theory of the problem $(P_{p}^{\lambda,\epsilon})$, and based on it, the IT algorithm is proposed to
solve the problem $(P_{p}^{\lambda,\epsilon})$ for all $0<p<1$. Indeed, we could get some much better results by choosing proper $p$, which is one of the advantages for our algorithm
compared with other methods. Numerical results also show that, for some proper $p$, our algorithm performs the best in some sparse signal recovery problems compared with some
state-of-art methods.
\end{abstract}

\section{Introduction}\label{section1}
Over the last decade, the compressed sensing \cite{candes1,candes2,donoho3} has attracted much attention in many science applications such as signal and image processing \cite{elad4},
medicine \cite{lust5}, astronomy \cite{bobin6}, seismology \cite{herrmann7}, and so on. The fundamental problem of compressed sensing is to recover a high-dimensional sparse
signal from a small number of linear measurements. In mathematics, it can be modeled into the following $\l_{0}$-minimization problem:
\begin{equation}\label{equ1}
(P_{0})\ \ \ \ \ \min_{x\in \mathbb{R}^{n}}\|x\|_{0}\ \ \mathrm{subject}\ \mathrm{to}\ \ Ax=b,
\end{equation}
where $A$ is a $m\times n$ real matrix of full row rank with $m\ll n$, $b$ is a nonzero real column vector of $m$-dimension, and $\|x\|_{0}$ is the so-called $\l_0$-norm of real
vector $x$, which counts the number of the non-zero entries in $x$. Unfortunately, although the $\l_{0}$-norm provides a very simple and essentially grasped notion of sparsity,
the problem $(P_{0})$ is truly a challenging non-convex optimization problem for which all known finite time algorithms have at least doubly exponential running times in both
theory and practice and is known to be NP-hard and is also NP-hard to approximate. The $\l_{1}$-norm minimization problem $(P_{1})$ is the most popular alternative:
\begin{equation}\label{equ2}
(P_{1})\ \ \ \ \ \min_{x\in \mathbb{R}^{n}}\|x\|_{1}\ \ \mathrm{subject}\ \mathrm{to}\ \ Ax=b,
\end{equation}
where $\|x\|_{1}=\sum_{i=1}^{n}|x_{i}|$. It is the tightest convex relaxation of the NP-hard problem $(P_{0})$ and many excellent theoretical and algorithmic works
(see, e.g., \cite{candes1,candes8,chen9,dau10,donoho11,donoho12,donoho13,donoho14,gri15,gold16,yang17,yin18,dau27}) have been proposed to solve the problem $(P_{1})$. However, as
the compact convex relaxation of the problem $(P_{0})$, the problem $(P_{1})$ may be suboptimal for recovering a real sparse signal, and its regularization problem tends to
lead to biased estimation by shrinking all the entries toward to zero simultaneously, and sometimes results in over-penalization.

With recent development of non-convex relaxation approach in sparse signal recovery problems, many researchers have shown that using the $\l_{p}$-norm $(0<p<1)$ to approximate
the $\l_{0}$-norm is a better choice than using the $\l_{1}$-norm (see, e.g.,\cite{cha19,cha20,fou21,lai22,che23,dau24,mou25,sun26,pen28,xu29}). Because of the
following relationship
\begin{equation}\label{equ3}
\|x\|_{0}=\lim_{p\rightarrow 0^{+}}\sum_{i=1}^{n}|x_{i}|^{p}=\lim_{p\rightarrow 0^{+}}\|x\|_{p}^{p}.
\end{equation}
The $\l_{p}$-norm $(0<p<1)$ minimization problem $(P_{p})$ seems to be the most popular choice to find the sparse signal, and the minimization takes the form
\begin{equation}\label{equ4}
(P_{p})\ \ \ \ \ \min_{x\in \mathbb{R}^{n}}\|x\|_{p}^{p}\ \ \mathrm{subject}\ \mathrm{to}\ \ Ax=b.
\end{equation}
It is important to emphasize that, in \cite{pen28}, the authors demonstrated that in every underdetermined linear system $Ax=b$ there corresponds a constant $p^{*}(A,b)>0$, which
is called NP/CMP equivalence constant, such that every solution to the problem $(P_{p})$ also solves the problem $(P_{0})$ whenever $0<p<p^{*}(A,b)$.

Different from the convex optimization problem $(P_{1})$, the non-convex optimization problem $(P_{p})$ is much more computational challenges, and is also NP-hard \cite{ge30}.
In \cite{dau31}, the iteration reweighted least squares minimization algorithm (IRLS algorithm in short) is proposed to solve the problem $(P_{p})$ for all $0<p<1$. The authors
proved that the rate of local convergence of this algorithm was superlinear and that the rate was faster for smaller $p$ and increased towards quadratic as $p\rightarrow 0$, and,
at each iteration, the solution of a least squares problem is required, of which the computational complexity is $\mathcal{O}$$(mn^{2})$.

On the other hand, some optimization methods have been proposed for its regularized model
\begin{equation}\label{equ5}
(P_{p}^{\lambda})\ \ \ \ \ \min_{x\in \mathbb{R}^{n}}\Big\{\|Ax-b\|_{2}^{2}+\lambda \|x\|_{p}^{p}\Big\}
\end{equation}
where $\lambda>0$ is the regularization parameter. Xu et al. \cite{xu32} considered the $\l_{\frac{1}{2}}$ regularization and proposed the iterative half thresholding algorithm
(Half algorithm in short) to solve problem $(P_{p}^{\lambda})$ when $p=\frac{1}{2}$.  The authors showed that $\l_{\frac{1}{2}}$ regularization could be fast solved by Half algorithm,
and the computational complexity is $\mathcal{O}(mn)$. On the basis of the Half algorithm, Cao et al. \cite{cao33} proposed an iterative $\l_{\frac{2}{3}}$ thresholding
algorithm to solve problem $(P_{p}^{\lambda})$ when $p=\frac{2}{3}$.

Although the computational complexity of Half algorithm and iterative $\l_{\frac{2}{3}}$ thresholding algorithm are lower than IRLS algorithm, they are limited to few
specific values of parameter $p$ $(p=\frac{1}{2}, \frac{2}{3})$.

In this paper, a modified $\l_{p}$-norm minimization problem is considered to approximate the problem $(P_{p}^{\lambda})$ for all $0<p<1$. In this new modified model, the
$\l_{p}$-norm $\|x\|_{p}^{p}$ is replaced by
\begin{equation}\label{equ6}
\sum_{i=1}^{n}\frac{|x_{i}|}{(|x_{i}|+\epsilon_{i})^{1-p}}
\end{equation}
where $\epsilon=(\epsilon_{1},\epsilon_{2},\cdots,\epsilon_{n})^{\top}\succ0$. With the change of parameter $\epsilon_{i}>0$,
we have
\begin{equation}\label{equ7}
\lim_{\epsilon_{i}\rightarrow0^{+}}\frac{|x_{i}|}{(|x_{i}|+\epsilon_{i})^{1-p}}\approx|x_{i}|^{p},        
\end{equation}
and the function (\ref{equ7}) interpolates the $\l_{p}$-norm of vector $x$:
\begin{equation}\label{equ8}
\lim_{\epsilon_{i}\rightarrow0^{+}}\sum_{i=1}^{n}\frac{|x_{i}|}{(|x_{i}|+\epsilon_{i})^{1-p}}\approx\|x\|^{p}_{p}.
\end{equation}
By this transformation, the problem $(P_{p}^{\lambda})$ could be approximated by the following mnodified $\l_{p}$-norm minimization problem
\begin{equation}\label{equ9}
(P_{p}^{\lambda,\epsilon})\ \ \ \ \ \min_{x\in \mathbb{R}^{n}}\Big\{\|Ax+b\|_{2}^{2}+\lambda \sum_{i=1}^{n}\frac{|x_{i}|}{(|x_{i}|+\epsilon_{i})^{1-p}}\Big\}.
\end{equation}

The rest of this paper is organized as follows. In Section \ref{section2}, an iterative thresholding (IT) algorithm is proposed to solve the problem $(P_{p}^{\lambda,\epsilon})$. The
convergence of the IT algorithm is established in Section \ref{section3}. In Section \ref{section4}, we present the experiments with a series of sparse signal recovery applications
to demonstrate the effectiveness of our algorithm. Some conclusion remarks are presented in Section \ref{section5}.

\section{The thresholding representation theory and algorithm for solving the problem $(P_{p}^{\lambda,\epsilon})$}\label{section2}
In this section, we firstly establish the thresholding representation theory of the problem $(P_{p}^{\lambda,\epsilon})$, which underlies the algorithms to be proposed.
Then, an iterative thresholding algorithm is proposed to solve the problem $(P_{p}^{\lambda,\epsilon})$ for all $p\in(0,1)$.

\subsection{Thresholding representation theory of $(P_{p}^{\lambda,\epsilon})$}\label{subsection2-1}
In this subsection, we establish the thresholding representation theory of the problem $(P_{p}^{\lambda,\epsilon})$, which underlies the algorithm to be proposed.

Before the analytic expression of the thresholding representation theory of the problem $(P_{p}^{\lambda,\epsilon})$, a crucial result need to be introduced for later use.

\begin{lemma}\label{lem1}{\rm(see \cite{dau10})}
For any $\lambda>0$ and $\alpha, \beta\in \mathbb{R}$, suppose that
\begin{equation}\label{equ10}
S_{\lambda,1}(\beta)\triangleq\arg\min_{\alpha\in \mathbb{R}}\{(\alpha-\beta)^{2}+\lambda|\alpha|\}.
\end{equation}
then the operator $S_{\lambda,1}(\beta)$ can be expressed by
\begin{equation}\label{equ11}
S_{\lambda,1}(\beta)=\mathrm{sign}(\beta)\cdot\max\{|\beta|-\frac{\lambda}{2},0\}.
\end{equation}
\end{lemma}

Nextly, we will show that the optimal solution to $(P_{p}^{\lambda,\epsilon})$ could be expressed as a thresholding operation.

For any fixed positive parameters $\lambda>0$, $\mu>0$, $a>0$ and $x, y\in \mathbb{R}^{n}$, let
\begin{equation}\label{equ12}
\mathcal{H}_{1}(x)=\|Ax-b\|_{2}^{2}+\lambda \sum_{i=1}^{n}\frac{|x_{i}|}{(|x_{i}|+\epsilon_{i})^{1-p}},
\end{equation}
and its surrogate function
\begin{equation}\label{equ13}
\begin{array}{llll}
\mathcal{H}_{2}(x,y)&=&\displaystyle\mu\|Ax-b\|_{2}^{2}+\lambda\mu \sum_{i=1}^{n}\frac{|x_{i}|}{(|y_{i}|+\epsilon_{i})^{1-p}}\\
&&-\mu\|Ax-Ay\|_{2}^{2}+\|x-y\|_{2}^{2}.
\end{array}
\end{equation}
It is clearly that $\mathcal{H}_{2}(x,x)=\mu \mathcal{H}_{1}(x)$ for all $\mu>0$.

\begin{theorem}\label{the1}
For any $\lambda>0$ and $0<\mu<\frac{1}{\|A\|_{2}^{2}}$. If $x^{\ast}$ is the optimal solution of $\displaystyle\min_{x\in \mathbb{R}^{n}}\mathcal{H}_{1}(x)$, then $x^{\ast}$ is also the optimal solution of $\displaystyle\min_{x\in \mathbb{R}^{n}}\mathcal{H}_{2}(x,x^{\ast})$, that is
$$\mathcal{H}_{2}(x^{\ast},x^{\ast})\leq \mathcal{H}_{2}(x,x^{\ast})$$
for any $x\in \mathbb{R}^{n}$.
\end{theorem}

\begin{proof}
By the definition of $\mathcal{H}_{2}(x, y)$, we have
\begin{eqnarray*}
\mathcal{H}_{2}(x,x^{\ast})&=&\mu\|Ax-b\|_{2}^{2}+\lambda\mu \sum_{i=1}^{n}\frac{|x_{i}|}{(|x^{\ast}_{i}|+\epsilon_{i})^{1-p}}\\
&&-\mu\|Ax-Ax^{\ast}\|_{2}^{2}+\|x-x^{\ast}\|_{2}^{2}\\
&\geq&\mu\|Ax-b\|_{2}^{2}+\lambda\mu \sum_{i=1}^{n}\frac{|x_{i}|}{(|x^{\ast}_{i}|+\epsilon_{i})^{1-p}}\\
&\geq&\mu \mathcal{H}_{1}(x^{\ast})\\
&=&\mathcal{H}_{2}(x^{\ast},x^{\ast}).
\end{eqnarray*}
This completes the proof.
\end{proof}

Theorem \ref{the1} implies that if $x^{\ast}$ is the optimal solution of $\displaystyle\min_{x\in \mathbb{R}^{n}}\mathcal{H}_{1}(x)$, then $x^{\ast}$ is also the optimal solution of $\displaystyle\min_{x\in \mathbb{R}^{n}}\mathcal{H}_{2}(x,x^{\ast})$.

\begin{theorem}\label{the2}
For any $\lambda>0$, $\mu>0$ and optimal solution $x^{\ast}$ of $\displaystyle\min_{x\in \mathbb{R}^{n}}\mathcal{H}_{1}(x)$, $\displaystyle\min_{x\in \mathbb{R}^{n}}\mathcal{H}_{2}(x,x^{\ast})$ is equivalent to
\begin{equation}\label{equ14}
\min_{x\in \mathbb{R}^{n}}\Big\{\|x-B_{\mu}(x^{\ast})\|_{2}^{2}+\lambda\mu \sum_{i=1}^{n}\frac{|x_{i}|}{(|x^{\ast}_{i}|+\epsilon_{i})^{1-p}}\Big\}
\end{equation}
where $B_{\mu}(x^{\ast})=x^{\ast}+\mu A^{\top}(b-Ax^{\ast})$.
\end{theorem}

\begin{proof}
By the definition, $\mathcal{H}_{2}(x,y)$ can be rewritten as
\begin{eqnarray*}
\mathcal{H}_{2}(x,x^{\ast})&=&\|x-(x^{\ast}-\mu A^{\top}Ax^{\ast}+\mu A^{\top}b)\|_{2}^{2}+\lambda\mu \sum_{i=1}^{n}\frac{|x_{i}|}{(|x^{\ast}_{i}|+\epsilon_{i})^{1-p}}\\
&&+\mu\|b\|_{2}^{2}-\|x^{\ast}-\mu A^{\top}Ax^{\ast}+\mu A^{\top}b\|_{2}^{2}+\|x^{\ast}\|_{2}^{2}-\mu\|Ax^{\ast}\|_{2}^{2}\\
&=&\|x-B_{\mu}(x^{\ast})\|_{2}^{2}+\lambda\mu\sum_{i=1}^{n}\frac{|x_{i}|}{(|x^{\ast}_{i}|+\epsilon_{i})^{1-p}}+\mu\|b\|_{2}^{2}-\|B_{\mu}(x^{\ast})\|_{2}^{2}\\
&&+\|x^{\ast}\|_{2}^{2}-\mu\|Ax^{\ast}\|_{2}^{2}
\end{eqnarray*}
which implies that $\displaystyle\min_{x\in \mathbb{R}^{n}}\mathcal{H}_{2}(x,x^{\ast})$ for any $\lambda>0$, $\mu>0$ is equivalent to
$$\min_{x\in \mathbb{R}^{n}}\Big\{\|x-B_{\mu}(x^{\ast})\|_{2}^{2}+\lambda\mu \sum_{i=1}^{n}\frac{|x_{i}|}{(|x^{\ast}_{i}|+\epsilon_{i})^{1-p}}\Big\}. $$
This completes the proof.
\end{proof}

Combining Theorem \ref{the1} and  Theorem \ref{the2}, we have the following corollary.
\begin{corollary}\label{cor1}
Let $x^{\ast}$ be the optimal solution of $(P_{p}^{\lambda,\epsilon})$. Then $x^{\ast}$ is also the optimal solution of the following minimization problem
\begin{equation}\label{equ15}
\min_{x\in \mathbb{R}^{n}}\Big\{\|x-B_{\mu}(x^{\ast})\|_{2}^{2}+\lambda\mu \sum_{i=1}^{n}\frac{|x_{i}|}{(|x^{\ast}_{i}|+\epsilon_{i})^{1-p}}\Big\}.
\end{equation}
\end{corollary}

In the following, we derive the most important conclusion in this paper, which underlies the algorithms to be proposed.
\begin{theorem}\label{the3}
Let $x^{\ast}\in \mathbb{R}^{n}$ be the optimal solution of the problem $(P_{p}^{\lambda,\epsilon})$. Then it can be given by
\begin{equation}\label{equ16}
\begin{array}{llll}
x^{\ast}_{i}&=&S_{\frac{\lambda\mu}{(|x_{i}^{\ast}|+\epsilon_{i})^{1-p}},1}([B_{\mu}(x^{\ast})]_{i})\\
&=&\displaystyle\mathrm{sign}([B_{\mu}(x^{\ast})]_{i})\cdot\max\Big\{|[B_{\mu}(x^{\ast})]_{i}|-\displaystyle\frac{\lambda\mu}{2(|x_{i}^{\ast}|+\epsilon_{i})^{1-p}}, 0\Big\}
\end{array}
\end{equation}
for $i=1,2,\cdots,n$, where $[B_{\mu}(x^{\ast})]_{i}$ represents the $i$-th component of vector $B_{\mu}(x^{\ast})$.
\end{theorem}

\begin{proof}
It is to see clear that the problem (\ref{equ15}) can be rewritten as
\begin{equation}\label{equ17}
\min_{x\in \mathbb{R}^{n}}\sum_{i=1}^{n}\Big\{(x_{i}-[B_{\mu}(x^{\ast})]_{i})^{2}+\lambda\mu \frac{|x_{i}|}{(|x^{\ast}_{i}|+\epsilon_{i})^{1-p}}\Big\}.
\end{equation}
Then, solving the problem (\ref{equ15}) reduces to
\begin{equation}\label{equ18}
\min_{x_{i}\in \mathbb{R}}\Big\{(x_{i}-[B_{\mu}(x^{\ast})]_{i})^{2}+\lambda\mu \frac{|x_{i}|}{(|x^{\ast}_{i}|+\epsilon_{i})^{1-p}}\Big\}.
\end{equation}
Together with Lemma \ref{lem1}, we immediately finish the proof.
\end{proof}

\subsection{Iterative thresholding algorithm for solving $(P_{p}^{\lambda,\epsilon})$}\label{subsection2-2}
With the thresholding representation (\ref{equ18}), the iterative thresholding (IT) algorithm for solving ($P_{p}^{\lambda,\epsilon}$) for all $p\in(0,1)$
can be naturally defined as
\begin{equation}\label{equ19}
\begin{array}{llll}
x^{k+1}_{i}&=&S_{\frac{\lambda\mu}{(|x_{i}^{k}|+\epsilon_{i})^{1-p}},1}([B_{\mu}(x^{k})]_{i})\\
&=&\displaystyle\mathrm{sign}([B_{\mu}(x^{k})]_{i})\cdot\max\Big\{|[B_{\mu}(x^{k})]_{i}|-\displaystyle\frac{\lambda\mu}{2(|x_{i}^{k}|+\epsilon_{i})^{1-p}}, 0\Big\}
\end{array}
\end{equation}
for $k=0,1,2,\cdots$, where $S_{\frac{\lambda\mu}{(|x_{i}^{\ast}|+\epsilon_{i})^{1-p}},1}([B_{\mu}(x^{k})]_{i})$ is obtained by replacing $\lambda$ with $\frac{\lambda\mu}{(|x_{i}^{\ast}|+\epsilon_{i})^{1-p}}$ in $S_{\lambda,1}([B_{\mu}(x^{k})]_{i})$.

In general, the quantity of the solution of a regularization problem depends seriously on the setting of the regularization parameter $\lambda>0$. However, the selection of proper
regularization parameter is a very hard problem. In IT algorithm, the cross-validation method (see \cite{xu32}) is accepted to choose the proper regularization parameter $\lambda>0$. Thus, the
IT algorithm will be adaptive and intelligent on the choice of regularization parameter $\lambda$. To make this selection clear, we suppose that the vector $x^{\ast}$ of sparsity $r$ is the
optimal solution to the problem $(P_{p}^{\lambda,\varepsilon})$. Without loss of generality, we assume that
$$|[B_{\mu}(x^{\ast})]_{1}|\geq|[B_{\mu}(x^{\ast})]_{2}|\geq\cdots\geq|[B_{\mu}(x^{\ast})]_{n}|.$$
By (\ref{equ16}), the following inequalities hold
$$|[B_{\mu}(x^{\ast})]_{i}|>\frac{\lambda\mu}{2(|x_{i}^{\ast}|+\epsilon_{i})^{1-p}}\ \Leftrightarrow \ i\in\{1,2,\cdots,r\},$$
$$|[B_{\mu}(x^{\ast})]_{i}|\leq\frac{\lambda\mu}{2(|x_{i}^{\ast}|+\epsilon_{i})^{1-p}}\ \Leftrightarrow \ i\in\{r+1,r+2,\cdots,m\},$$
which implies
$$\max_{i\in\{r+1,r+2,\cdots,n\}}\frac{2|[B_{\mu}(x^{\ast})]_{i}|(|x_{i}^{\ast}|+\epsilon_{i})^{1-p}}{\mu}\leq\lambda$$
and
$$\lambda<\min_{i\in\{1,2,\cdots,r\}}\frac{2|[B_{\mu}(x^{\ast})]_{i}|(|x_{i}^{\ast}|+\epsilon_{i})^{1-p}}{\mu}.$$
For the sake of simplicity, we set
\begin{equation}\label{equ20}
\lambda\in\bigg[\displaystyle\frac{2|[B_{\mu}(x^{\ast})]_{r+1}|(\lceil x^{\ast}\rfloor_{r+1}+\lceil\epsilon\rfloor_{r+1})^{1-p}}{\mu}, \frac{2|[B_{\mu}(x^{\ast})]_{r}|(\lceil x^{\ast}\rfloor_{r}+\lceil\epsilon\rfloor_{r})^{1-p}}{\mu}\bigg)
\end{equation}
where $\lceil x\rfloor$ is the nonincreasing rearrangement of the vector $|x|$ for which
$$\lceil x\rfloor_{1}\geq \lceil x\rfloor_{2}\geq\cdots\geq \lceil x\rfloor_{n}\geq0$$
and there is a permutation $\pi:[n]\rightarrow[n]$ with $\lceil x\rfloor_{i}=|x_{\pi(i)}|$ for all $i\in[n]$.

In practice, we approximate $B_{\mu}(x^{\ast})$ (resp., $\tilde{x}^{\ast}$) by $B_{\mu}(x^{k})$ (resp., $\tilde{x}^{k}$) in (\ref{equ20}), and a choice of $\lambda$ is
\begin{equation}\label{equ21}
\begin{array}{llll}
\lambda_{k}\in\bigg[\displaystyle\frac{2|[B_{\mu}(x^{k})]_{r+1}|(\lceil x^{k}\rfloor_{r+1}+\lceil\epsilon\rfloor_{r+1})^{1-p}}{\mu},\frac{2|[B_{\mu}(x^{k})]_{r}|(\lceil x^{k}\rfloor_{r}+\lceil\epsilon\rfloor_{r})^{1-p}}{\mu}\bigg).
\end{array}
\end{equation}
We can then take
$$\lambda_{k}=\frac{2}{\mu}\Big[(1-\alpha)|[B_{\mu}(x^{k})]_{r+1}|(\lceil x^{k}\rfloor_{r+1}+\lceil\epsilon\rfloor_{r+1})^{1-p}+\alpha|[B_{\mu}(x^{k})]_{r}|(\lceil x^{k}\rfloor_{r}+\lceil\epsilon\rfloor_{r})^{1-p}\Big].$$
with any $\alpha\in(0,1)$. Taking $\alpha=0$, this leads to a most reliable choice of $\lambda$ specified by
\begin{equation}\label{equ22}
\lambda=\lambda_{k}=\frac{2|[B_{\mu}(x^{k})]_{r+1}|(\lceil x^{k}\rfloor_{r+1}+\lceil\epsilon\rfloor_{r+1})^{1-p}}{\mu}
\end{equation}
in each iteration.

\begin{algorithm}[h!]
\caption{: IT Algorithm}
\label{alg:B}
\begin{algorithmic}
\STATE {\textbf{Initialize}: Choose $x^{0}\in \mathbb{R}^{n}$, $\epsilon\in \mathbb{R}^{n}_{+}$, $\mu=\frac{1-\eta}{\|A\|_{2}^{2}}$ $(\eta\in(0,1))$ and $p\in(0,1)$;}
\STATE \ \ {$k=0$;}
\STATE \ \ {\textbf{while} not converged \textbf{do}}
\STATE \ \ \ \ \ \ \ \ {$B_{\mu}(x^{k})=x^{k}+\mu A^{T}(y-Ax^{k})$;}
\STATE \ \ \ \ \ \ \ \ {$\lambda=\lambda_{k}=\frac{2|[B_{\mu}(x^{k})]_{r+1}|(\lceil x^{k}\rfloor_{r+1}+\lceil\epsilon\rfloor_{r+1})^{1-p}}{\mu}$;}
\STATE \ \ \ \ \ \ \ \ {for\ $i=1:n$}
\STATE \ \ \ \ \ \ \ \ \ \ \ {$x^{k+1}_{i}=\mathrm{sign}([B_{\mu}(x^{k})]_{i})\cdot\max\Big\{|[B_{\mu}(x^{k})]_{i}|-\frac{\lambda\mu}{2(|x_{i}^{k}|+\epsilon_{i})^{1-p}},  0\Big\}$}
\STATE\ \ {$k\rightarrow k+1$}
\STATE\ \ {\textbf{end while}}
\STATE{\textbf{return}: $x^{k+1}$}
\end{algorithmic}
\end{algorithm}

\section{Convergence analysis of the IT algorithm}\label{section3}

In this section, the convergence of the IT algorithm is established under some specific conditions. It is necessary to emphasize that the ideas for the prove
of the convergence of the IT algorithm are mainly inspired by the former work of Xu et al. \cite{xu32}.

\begin{theorem}\label{the4}
Let $\{x^{k}\}$ be the sequence generated by iteration (\ref{equ19}) with the step size $\mu$ satisfying $0<\mu<\frac{1}{\|A\|_{2}^{2}}$.
Then the sequence $\mathcal{H}_{1}(x^{k})$ is decreasing and converging to $\mathcal{H}_{1}(x^{\ast})$, where $x^{\ast}$ is a limit point of minimization sequence $\{x^{k}\}$.
\end{theorem}

\begin{proof}
According to the proof of Theorem \ref{the1}, we have
$$\mathcal{H}_{2}(x^{k+1},x^{k})=\min_{x\in \mathbb{R}^{n}} \mathcal{H}_{2}(x, x^{k}).$$
Combined with the definition of $\mathcal{H}_{1}(x)$ and $\mathcal{H}_{2}(x,y)$, we have
\begin{eqnarray*}
\mathcal{H}_{1}(x^{k+1})&=&\frac{1}{\mu}[\mathcal{H}_{2}(x^{k+1}, x^{k})-\|x^{k+1}-x^{k}\|_2^{2}]+\|Ax^{k+1}-Ax^{k}\|_2^2.
\end{eqnarray*}
Since $0<\mu<\frac{1}{\|A\|_{2}^{2}}$, we get
\begin{equation}\label{equ23}
\begin{array}{llll}
\mathcal{H}_{1}(x^{k+1})&=&\displaystyle\frac{1}{\mu}[\mathcal{H}_{2}(x^{k+1},x^{k})-\|x^{k+1}-x^{k}\|_2^2]+\|Ax^{k+1}-Ax^{k}\|_2^2\\
&\leq&\displaystyle\frac{1}{\mu}[\mathcal{H}_{2}(x^{k},x^{k})-\|x^{k+1}-x^{k}\|_2^2]+\|Ax^{k+1}-Ax^{k}\|_2^2\\
&=&\displaystyle \mathcal{H}_{1}(x^{k})-\frac{1}{\mu}\|x^{k+1}-x^{k}\|_2^2+\|Ax^{k+1}-Ax^{k}\|_2^2\\
&\leq&\mathcal{H}_{1}(x^{k}).
\end{array}
\end{equation}
That is, the sequence $\{x^{k}\}$ is a minimization sequence of function $\mathcal{H}_{1}(x)$, i.e., $\mathcal{H}_{1}(x^{k+1})\leq \mathcal{H}_{1}(x^{k})\ \ \mathrm{for}\ \mathrm{all}\ \ k\geq0$,
and the sequence $\{\mathcal{H}_{1}(x^{k})\}$ monotonically decreases to a fixed value $\mathcal{H}_{1}^{\ast}$. Due to the bound of $x^{k}\in \{x\in \mathbb{R}^{n}: \mathcal{H}_{1}(x)\leq \mathcal{H}_{1}(x^{0})\}$, the sequence $\{x^{k}\}$ is bounded. So, the sequence $\{x^{k}\}$ exists a limit point $x^{\ast}$. According to the continuity of $\mathcal{H}_{1}(x)$ and the
monotonicity of $\mathcal{H}_{1}(x^{k})$, it then follows that $\mathcal{H}_{1}^{\ast}=\mathcal{H}_{1}(X^{\ast})$. This completes the proof.
\end{proof}

\begin{theorem}\label{the5}
The sequence $\{x^{k}\}$ generated by iteration (\ref{equ19}) with the step size $\mu$ satisfying $0<\mu<\frac{1}{\|A\|_{2}^{2}}$ is asymptotically regular, i.e., $\lim_{k\rightarrow\infty}\|x^{k+1}-x^{k}\|_{2}^{2}=0$.
\end{theorem}

\begin{proof}
Let $\theta=1-\mu\|A\|_{2}^{2}$, we can get that $\theta\in(0, 1)$ and
\begin{equation}\label{equ24}
\mu\|A(x^{k+1}-x^{k})\|_2^2\leq(1-\theta)\|x^{k+1}-x^{k}\|_2^2.
\end{equation}
By (\ref{equ23}), we have
\begin{equation}\label{equ25}
\displaystyle \frac{1}{\mu}\|x^{k+1}-x^{k}\|_2^2-\|Ax^{k+1}-Ax^{k}\|_2^2\leq \mathcal{H}_{1}(x^{k})-\mathcal{H}_{1}(x^{k+1}).
\end{equation}
Combing (\ref{equ24}) and (\ref{equ25}), we get
\begin{eqnarray*}
\sum_{k=0}^{N}\{\|x^{k+1}-x^{k}\|_2^2\}&\leq&\frac{1}{\theta}\sum_{k=0}^{N}\{\|x^{k+1}-x^{k}\|_2^2\}-\frac{1}{\theta}\sum_{k=0}^{N}\{\mu\|Ax^{k+1}-Ax^{k}\|_2^2\}\\
&\leq&\frac{\mu}{\theta}\sum_{k=0}^{N}\{\mathcal{H}_{1}(x^{k})-\mathcal{H}_{1}(x^{k+1})\}\\
&=&\frac{\mu}{\theta}(\mathcal{H}_{1}(x^{0})-\mathcal{H}_{1}(x^{N+1}))\\
&\leq&\frac{\mu}{\theta}\mathcal{H}_{1}(x^{0}).
\end{eqnarray*}
Thus, the series $\sum_{k=0}^{\infty}\|x^{k+1}-x^{k}\|_2^2$ is convergent, which means that
$$\|x^{k+1}-x^{k}\|_2^2\rightarrow 0 \ \ as \ \ k\rightarrow\infty.$$
This completes the proof.
\end{proof}

\begin{theorem} \label{the5}
Let $\{x^{k}\}$ be a sequence generated by iteration (\ref{equ19}) with the step size $\mu$ satisfying $0<\mu<\frac{1}{\|A\|_{2}^{2}}$. Then any accumulation point of $\{x^{k}\}$
is a stationary point of the problem $(P_{p}^{\lambda,\epsilon})$.
\end{theorem}

\begin{proof}
Let $\{x^{k_{j}}\}$ be a convergent subsequence of sequence $\{x^{k}\}$, and denote $x^{\ast}$ as the limit point of subsequence $\{x^{k_{j}}\}$, i.e.,
\begin{equation}\label{equ26}
x^{k_{j}}\rightarrow x^{\ast}\ \ \mathrm{as}\ \ j\rightarrow \infty.
\end{equation}
By following triangle inequality
$$\|x^{k_{j}+1}-x^{\ast}\|_{2}\leq\|x^{k_{j}+1}-x^{k_{j}}\|_{2}+\|x^{k_{j}}-x^{\ast}\|_{2}\rightarrow 0, \ \ k_{j}\rightarrow \infty,$$
we derive
\begin{equation}\label{equ27}
x^{k_{j}+1}\rightarrow x^{\ast}\ \ \mathrm{as}\ \ k_{j}\rightarrow \infty.
\end{equation}
Combing iteration (\ref{equ19}) and Theorem \ref{the2}, we can get that
\begin{eqnarray*}
\|x^{k_{j}+1}-B_{\mu}(x^{k_{j}})\|_{2}^{2}+\lambda\mu \sum_{i=1}^{N}\frac{|x_{i}^{k_{j}+1}|}{(|x_{i}^{k_{j}}|+\epsilon_{i})^{1-p}}\leq\|x-B_{\mu}(x^{k_{j}})\|_{2}^{2}+\lambda\mu \sum_{i=1}^{N}\frac{|x_{i}|}{(|x_{i}^{k_{j}}|+\epsilon_{i})^{1-p}}.
\end{eqnarray*}
Taking limit of the sequence $\{x^{k_{j}+1}\}$ and $\{x^{k_{j}}\}$ in above inequality, we can immediately get that
\begin{eqnarray*}
\|x^{\ast}-B_{\mu}(x^{\ast})\|_{2}^{2}+\lambda\mu \sum_{i=1}^{N}\frac{|x_{i}^{\ast}|}{(|x_{i}^{\ast}|+\epsilon_{i})^{1-p}}\leq\|x-B_{\mu}(x^{\ast})\|_{2}^{2}+\lambda\mu \sum_{i=1}^{N}\frac{|x_{i}|}{(|x_{i}^{\ast}|+\epsilon_{i})^{1-p}},
\end{eqnarray*}
which means that $x^{\ast}$ minimizes the following function
\begin{equation}\label{equ28}
\|x-B_{\mu}(x^{\ast})\|_{2}^{2}+\lambda\mu \sum_{i=1}^{n}\frac{|x_{i}|}{(|x_{i}^{\ast}|+\epsilon_{i})^{1-p}},
\end{equation}
and we can conclude that
$$x^{\ast}_{i}=S_{\frac{\lambda\mu}{(|x_{i}^{\ast}|+\epsilon_{i})^{1-p}},1}([B_{\mu}(x^{\ast})]_{i})$$
for all $i=1,2,\cdots,n$. This completes the proof.
\end{proof}

\section{Numerical experiments}\label{section4}
In this section, we carry out a series of simulations to demonstrate the performance of IT algorithm. The iterative soft thresholding algorithm (Soft algorithm)\cite{dau27},
iterative half thresholding algorithm (Half algorithm)\cite{xu32} and our IT algorithm are compared in these experiments. The experiments are all performed on a Lenovo-PC with
an Intel(R) Core(TM) i7-6700 CPU @ 3.40GHZ with 16GB of RAM running Microsoft Windows 7.

To show the success rate of these three algorithms in recovering a signal with the different cardinality for a given measurement matrix $A\in \mathbb{R}^{256\times 1024}$
with entries independently drawn by random from a Gaussian distribution, $\mathcal{N}(0,1)$. By randomly generating sparse vectors $x_{0}$, we generate vectors $b$, and we
know the sparsest solution to the linear system $Ax_{0}=b$. The stopping criterion is defined as
$$\frac{\|x^{k+1}-x^{k}\|_{2}}{\|x^{k}\|_{2}}\leq \mathrm{Tol}$$
where $x^{k+1}$ and $x^{k}$ are numerical results from two continuous iterative steps and $\mathrm{Tol}$ is a given small number. The success is measured by computing
the relative error (RE):
$$\mathrm{RE}=\frac{\|x^{\ast}-x_{0}\|_{2}}{\|x_{0}\|_{2}}$$
to indicate a perfect recovery of the original sparse vector $x_{0}$. In our experiments, we set to $\mathrm{Tol}=10^{-8}$. Moreover, we also find
that the the quantity of the solution of the IT algorithm also depends seriously on the setting of the parameter $\epsilon$, and in our experiments, we set
$$\epsilon_{i}=\max\{0.7|[\mu A^{\top}(b-Ax^{k})]_{i}|, 10^{-3}\}.$$
All of our experiments, we repeatedly perform 20 tests and present average results.

\begin{figure}[h!]
 \centering
 \includegraphics[width=0.65\textwidth]{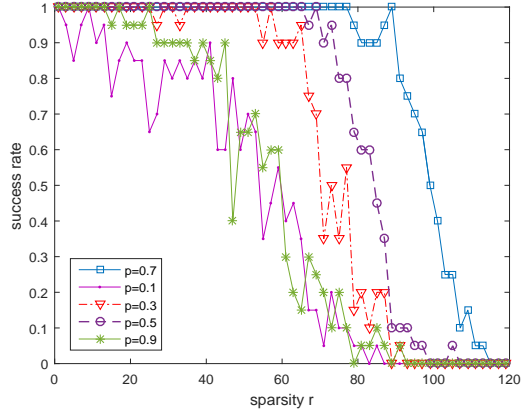}
\caption{The success rate of the IT algorithm in the recovery of a sparse signal with different sparsity for some different $p$.}
\label{fig:1}       
\end{figure}

The graphs presented in Figure \ref{fig:1} show the performances of the IT algorithm in recovering the true (sparsest) signals with different $p$. Comparing these performances
we can find that $p=0.7$ is the best strategy. The graphs presented in Figure \ref{fig:2} show the success rate of Soft algorithm, Half algorithm and IT algorithm in recovering
the true (sparsest) signals. We can see that IT algorithm can exactly recover the ideal signal until $r$ is around $78$, Half algorithm's counterpart is around $70$ and Soft
algorithm's counterpart is around $38$. As we can see, the IT algorithm has the best performance, with Half algorithm as the second.

\begin{figure}[h!]
 \centering
 \includegraphics[width=0.65\textwidth]{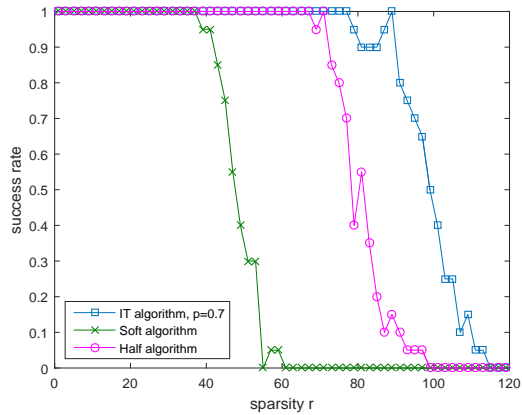}
\caption{The comparison of three algorithms in the recovery of a sparse signal with different sparsity.}
\label{fig:2}       
\end{figure}

\section{Conclusion}\label{section5}
In this paper, we replace the $\ell_{p}$-norm $\|x\|_{p}^{p}$ with an approximate function $\sum_{i=1}^{n}\frac{|x_{i}|}{(|x_{i}|+\epsilon_{i})^{1-p}}$. With change the parameter
$\epsilon>0$, this approximate function would like to interpolate the $\l_{p}$-norm $\|x\|_{p}^{p}$. By this transformation, we translated the $\l_{p}$-norm regularization
minimization problem $(P_{p}^{\lambda})$ into a variant $\l_{p}$-norm regularization minimization $(P_{p}^{\lambda,\epsilon})$. We develop the thresholding representation theory of
the problem $(P_{p}^{\lambda,\epsilon})$. Based on it, the IT algorithm is proposed to solve the problem $(P_{p}^{\lambda,\epsilon})$. Numerical results show that our algorithm
performs the best in sparse signal recovery problems compared with some state-of-art methods.

\begin{acknowledgements}
The authors would like to thank the reviewers and editors for their useful comments which significantly improve this paper. This research was supported by the
National Natural Science Foundation of China under the grants 11771347, 91730306, 41390454.
\end{acknowledgements}

\end{document}